\documentclass[12pt]{amsart}

\usepackage{amsmath}
\usepackage{amsfonts}
\usepackage{amssymb}
\usepackage{amsthm}
\usepackage{longtable}
\usepackage[mathscr]{eucal}\catcode`\@=11

\usepackage{url}
\usepackage{enumerate}

\theoremstyle{plain}
\newtheorem{theorem}{Theorem}[section]
\newtheorem{lemma}{Lemma}[section]

\begin{document}

\baselineskip=17pt


\title[On equal values of products and power sums...]{On equal values of products and power sums of consecutive elements in an arithmetic progression}

\author[A. Bazs\'o, D. Kreso, F. Luca and \'A. Pint\'er]{A. Bazs\'o, D. Kreso, F. Luca, \'A. Pint\'er, and Cs. Rakaczki}

\address{A. Bazs\'o \newline
\indent Institute of Mathematics \newline
\indent University of Debrecen \newline
\indent P.O. Box 400, H-4002 Debrecen, Hungary \newline
\indent and \newline
\indent ELKH-DE Equations, Functions, Curves and their Applications Research Group}
\email{bazsoa@science.unideb.hu}

\address{D. Kreso \newline
\indent Institute f\"ur Mathematik \newline
\indent Technische Universit\"at Graz \newline
\indent Steyrergasse 30, 8010 Graz Austria}
\email{kreso@math.tugraz.at}

\address{F. Luca \newline
\indent School of Mathematics \newline
\indent Wits University \newline
\indent 1 Jan Smuts, Brammfontein, 2000 Johannesburg, South Africa, \newline
\indent Research Group in Algebraic Structures and Applications \newline
\indent King Abdulaziz University \newline
\indent Abdulah Sulayman, Jeddah 22254, Saudi Arabia, \newline
\indent and \newline
\indent Mathematical Institute \newline
\indent UNAM Ap. Postal 61-3 (Xangari) \newline
\indent CP 58 089. Morelia, Michoac\'an, Mexico}
\email{florian.luca@wits.ac.za}

\address{\'A. Pint\'er \newline
\indent Institute of Mathematics \newline
\indent U
niversity of Debrecen \newline
\indent P.O. Box 400, H-4002 Debrecen, Hungary}
\email{apinter@science.unideb.hu}

\address{Cs. Rakaczki\newline
\indent Institute of Mathematics\newline
\indent University of Miskolc \newline
\indent H-3515 Miskolc Campus, Hungary}
\email{ matrcs@uni-miskolc.hu}

\thanks{}

\date{}

\begin{abstract}
In this paper we study the Diophantine equation
\begin{align*}
b^k + \left(a+b\right)^k + &\left(2a+b\right)^k + \ldots + \left(a\left(x-1\right) + b\right)^k = \\
&y\left(y+c\right) \left(y+2c\right) \ldots \left(y+ \left(\ell-1\right)c\right),
\end{align*}
where $a,b,c,k,\ell$ are given  integers under natural conditions. We prove some effective results for special values for $c,k$ and $\ell$ and obtain a general ineffective result based on Bilu-Tichy method.
\end{abstract}

\maketitle

\section{Introduction}

The polynomials
\begin{equation}
S_{a,b}^k \left(x\right) = b^k + \left(a+b\right)^k + \left(2a+b\right)^k + \ldots + \left(a\left(x-1\right) + b\right)^k \label{pol:skabx}
\end{equation}
and
\begin{equation}
R_c^{\ell} \left(x\right) = x\left(x+c\right) \left(x+2c\right) \ldots \left(x+ \left(\ell-1\right)c\right),
\end{equation}
are natural generalizations of the widely studied polynomials $S_k (x) = S_{1,0}^k (x)$ and $R_{\ell} (x) =R_1 ^{\ell} (x)$, respectively.
Various Diophantine equations concerning $R_{\ell} (x)$ and $S_k (x)$ have been extensively investigated. See e.g. \cite{BBKPT} and the references given there.

It is easy to see that involving Bernoulli polynomials, the polynomial defined above by \eqref{pol:skabx} can be rewritten as
\begin{equation} \label{eq:mainI}
S_{a,b}^k \left(x\right) =  \frac{a^k}{k+1} \left(B_{k+1} \left(x+ \frac{b}{a}\right) - B_{k+1} \left(\frac{b}{a}\right)\right).
\end{equation}

In \cite{BBKPT}, Bilu, Brindza, Kirschenhofer, Pint\'er and Tichy proved that for $k \geq 1, \ell \geq 2$, and $(k,\ell) \neq (1,2)$, the equation $S_k (x) = R_{\ell} (y)$ has at most finitely many integer solutions. They also proved a similar result for the equation $S_k (x) = S_{\ell} (y)$. Both of these results were ineffective, since their proofs were mainly based on the general finiteness criterion of Bilu and Tichy \cite{BiluTichy} for Diophantine equations of the form $f(x) = g(y)$. In certain special cases they also proved effective finiteness results for the corresponding equations.

In our earlier paper \cite{BKLP}, using a slightly modified approach, we generalized the above result of Bilu, Brindza, Kirschenhofer, Pint\'er and Tichy \cite{BBKPT} concerning the equation $S_k (x) = S_{\ell} (y)$ by proving that the more general equation $S_{a,b}^k (x) = S_{c,d}^{\ell} (y)$ has at most finitely many solutions in rational integers $x,y$. This theorem is also ineffective, but it is made effective in some special cases.

The purpose of this paper is to study the equation
\begin{equation} \label{eq:RS}
S_{a,b}^k (x) = R_{c}^{\ell} (y)
\end{equation}
in integers $x,y$.

As first result we prove a generalization of original Sch\"affer problem on the power values of power sums.

\begin{theorem} \label{thm:eff1}
Let $a,b$ be rational integers with $a>0$ and suppose that $c=0$. We consider equation
\begin{equation}\label{eq:c=0}
S_{a,b}^k (x) =y^{\ell}
\end{equation}
in integers $x,y>1$ and $\ell\geq 2$. If $(k,a,b)\neq (1,2,1 )$ then $\ell<C_1$ where $C_1$ is an effectively computable constant depending only on $k,a$ and $b$. Further, apart from the cases
$$(k,\ell,a,b)\in\{ (1,2,a,0),(3,2,a,0),(3,2,2,1),(3,4,a,0), (5,2,a,0) \}$$
equation (\ref{eq:c=0}) implies $\max(|x|,|y|)<C_2$ where $C_2$ is an effectively computable constant depending on $k,\ell,a$ and $b$.
\end{theorem}

In degenerate case  $(k,a,b)\neq (1,2,1 )$ we have $x^2=y^{\ell}$, thus we can not give an upper bound for $\ell$. From the 5 exceptional cases 4 are the well-known examples by Sch\"affer ($b=0)$, and if  $(k,\ell,a,b)=(3,2,2,1)$ we get the equation $x^2(2x^2-1)=y^2$, and the theory of Pell equations yields infinitely many integer solutions in $x,y$. We remark that Bazs\'o \cite{B} considered a more general case for shifted power values of power sums.

One can treat the cases when the parameters $k$ or $\ell$ are small. Set
$I_1=\{(1,2),(1,4),(3,2),(3,4)     \}$ and $I_2=I_1\cup \{(5,2),(5,4)\}$. We obtain

\begin{theorem} \label{thm:eff2}

 Let $\ell\geq 2$ be a rational integer and  $k\in\{1,3\}$ with $(k,\ell)\notin I_1$, or $k\geq 1$ is a rational integer and $\ell\in \{2,4\}$ with $(k,\ell)\notin I_2$. Equation (\ref{eq:RS}) has only finitely many solutions in integers $x$ and $y$, and $\max(|x|,|y|)$ is bounded by an effectively computable constant depending on $a,b,c$ and $\ell$ or $k$, respectively.
\end{theorem}

When $(k,\ell)=(1,2)$ from (\ref{eq:RS}) we have the Pellian equation
$$
(2ax+2b-a)^2-2a(2y+c)^2=(2b-a)^2-2ac^2.
$$
Now, there are infinitely many solutions in positive integers $a,b,c$ of the equation $(2b-a)^2-2c^2=1$, and similarly for fixed triplet $(a,b,c)$ there exist infinitely many solutions $x,y$ for the previous Pellian equation.

Focusing on the exceptional cases we have

\begin{theorem} \label{thm:eff3}

Apart from the cases
\begin{itemize}
 \item $(k,\ell,a,b,c)=(1,4,2,b,c)$ with $b=\pm 2c^2+1$,
\item $(k,\ell,a,b,c)=(3,2,a,b,c)$ with $\frac{c^2}{a^3}-B_4\left( \frac{b}{a}    \right)=\frac{1}{30}$ or $-\frac{7}{240}$,
\item $(k,\ell,a,b,c)=(3,4,1,b,c)$ with $b(b-1)=2c^2$,
\item $(k,\ell,a,b,c)=(5,2,a,b,c)$ with $\frac{3c^2}{2a^5}-B_6\left( \frac{b}{a}    \right)=-\frac{1}{42}$ or $-\frac{1}{189}$,
 and
\item $ (k,\ell,a,b,c)=(5,4,a,b,c)$ with $\frac{6c^4}{a^5}-B_6\left( \frac{b}{a}    \right)=-\frac{1}{42}$ or $-\frac{1}{189},$
\end{itemize}
the equations
$$S_{a,b}^1 (x) = R_{c}^{4} (y), S_{a,b}^3 (x) = R_{c}^{2} (y),S_{a,b}^3 (x) = R_{c}^{4} (y),S_{a,b}^5 (x) = R_{c}^{2} (y),$$
and $S_{a,b}^5 (x) = R_{c}^{4} (y)$, respectively,
in integers $x,y$ imply $\max(|x|,|y|)<C_3,$ where $C_3$ is an effectively computable constant depending on $a,b$ and $c$.

\end{theorem}

Our main result is the following general analogue of Theorem 1.1 in \cite{BBKPT}.
\begin{theorem} \label{thm:main}
Let $k,\ell$ be rational integers with $k\geq 2, k\notin\{3,5\}$  and $\ell=3$ or $\ell\geq 5$. Then for all nonzero integers $a,b,c$ with $\gcd(a,b)=1$ equation (\ref{eq:RS}) has only finitely many solutions $(x,y)$.
\end{theorem}

\section{Auxiliary results}

We denote by $\mathbb{C}[x]$ the ring of polynomials in the variable $x$ with complex coefficients. A decomposition of a polynomial $F(x) \in \mathbb{C}[x]$ is an equality of the following form
$$
F(x) = G_1 (G_2 (x)) \ \ \ (G_1 (x), G_2 (x) \in \mathbb{C}[x]),
$$
which is nontrivial if
$$
\deg G_1 (x) > 1 \ \ \ \text{and} \ \ \ \deg G_2 (x) > 1.
$$
Two decompositions $F(x) = G_1 (G_2 (x))$ and $F(x) = H_1 (H_2 (x))$ are said to be equivalent if there exists a linear polynomial $\ell (x) \in \mathbb{C}[x]$ such that $G_1 (x) = \ell (H_1 (x))$ and $H_2 (x) = \ell (G_2 (x))$. The polynomial $F(x)$ is called decomposable if it has at least one nontrivial decomposition; otherwise it is said to be indecomposable.

Bazs\'o, Pint\'er and Srivastava \cite{BPS} recently proved the following theorem about the decomposition of the polynomial $S_{a,b}^k \left(x\right)$.

\begin{lemma} \label{thm:BPS}
The polynomial $S_{a,b}^k \left(x\right)$ is indecomposable for even $k$. If $k=2v-1$ is odd, then any nontrivial decomposition of $S_{a,b}^k \left(x\right)$ is equivalent to the following decomposition:
\begin{equation}
S_{a,b}^k \left(x\right) = \widehat{S}_v \left(\left(x+\frac{b}{a} - \frac{1}{2}\right)^2\right).
\end{equation}
\end{lemma}

\begin{proof}[Proof of Lemma \ref{thm:BPS}]
This is Theorem 2 of \cite{BPS}.
\end{proof}

For classifying the decompositions of the polynomial $R_c^{\ell} (x)$ we need the following lemma.

\begin{lemma} \label{lemma:Rlx}
The polynomial $R_{\ell} (x) =R_1 ^{\ell} (x)$  is indecomposable if  $k$  is odd. If  $\ell=2m$ is
even then any nontrivial decomposition of  $R_k(x)$  is equivalent to
\begin{equation}
R_{\ell}(x)=\widehat{R}_m((x+(\ell-1)/2)^2),
\end{equation}
where
$$\widehat{R}_m(x)=\left(x-\frac{1}{4}\right)\left(x-\frac{9}{4}\right)\cdots \left(x-\frac{(2m-1)^2}{4}\right).$$
In particular, the polynomial $\widehat{R}_m(x)$ is indecomposable for any $m$.
\end{lemma}

\begin{proof}[Proof of Lemma \ref{lemma:Rlx}]
See Theorem 4. 3 in \cite{BBKPT}.
\end{proof}

The proof of the general case is based  on the previous lemma and on the easy observation

\begin{equation} \label{eq:obser}
R_c^{\ell} (x)=c^{\ell} R_{\ell} \left(\frac{x}{c}\right).
\end{equation}
\begin{lemma} \label{lemma:Rclx}
The polynomial $R_c^{\ell} (x)$ is indecomposable if $\ell$ is odd. If $\ell=2m$ is even, then any nontrivial decomposition of $R_c^{\ell} (x)$ is equivalent to
$$
R_c^{\ell} (x) = \widehat{R}_c^m \left(\left(x + \frac{(\ell-1)c}{2}\right)^2\right),
$$
where
$$
\widehat{R}_c^m (x) = \left(x - \frac{c^2}{4}\right) \left(x - \frac{9c^2}{4}\right) \ldots \left(x - \frac{((2m-1)c)^2}{4}\right).
$$
and the polynomial $\widehat{R}_c^m (x)$ is indecomposable for any $m$.
\end{lemma}

\begin{proof}[Proof of Lemma \ref{lemma:Rclx}]
Let $\ell$ be an odd integer with $\ell\geq 1$. On supposing the contrary we obtain
$$R_c^{\ell} (x)=f_1(f_2(x)),$$
where $\deg f_1>1$ and $\deg f_2>1$. Using (\ref{eq:obser}) we have
$$c^{\ell} R_{\ell} \left(\frac{x}{c}\right)=f_1(f_2(x))$$
and
$$R_{\ell}(x)=\frac{1}{c^{\ell}}f_1(f_2(cx))$$
which is a contradiction. In the even case, from Lemma \ref{lemma:Rlx} and (\ref{eq:obser}) we get
$$f_2(x)=\left(\frac{x}{c}+\frac{\ell-1}{2}\right)^2$$
and our lemma is proved.
\end{proof}

Our next lemma provides information on the structure of the zeros of Bernoulli polynomials.

\begin{lemma} \label{eff:1}
(i) For every $d\in \mathbb Q$ and rational integer $k\geq 3$ the polynomial $B_k(x)+d$ has at least three simple zeros apart from the cases $(k,d)\in\{(4,\frac{1}{30}),(4,-\frac{7}{240}),(6,-\frac{1}{42}),(6,-\frac{1}{189})\}$.

(ii) For every $d\in \mathbb Q$ and rational integer $k\geq 7$, the polynomial $B_k(x)+d$ has at least one complex nonreal zero.

(iii) The zeros of $B_k(x)$ are all simple.

\end{lemma}

\begin{proof}[Proof of Lemma \ref{eff:1}]
For $d=0$ and odd values of $k\geq 3$ Part (i) is a consequence of a theorem by Brillhart \cite[Corollary of Theorem 6]{Bril}. For non-zero rational $d$ and odd $k$ with $k\geq 3$ and for even values of $k\geq 4$ our lemma follows from \cite[Theorem]{pr}, and \cite[Theorem 2.3]{raka} and the subsequent remarks, respectively.

For (ii) assume that all the zeros of $B_k(x)+d$ are real. Then also all the zeros of its derivative
$$(B_k(x)+d)'=kB_{k-1}(x)$$
are real. By induction, all the roots of $B_{k-1}(x), B_{k-2}(x),\ldots $ are real. Since $k\geq 7$ and $B_6(x)$ has a complex nonreal root, we obtain a contradiction. Part (iii) was proved in \cite{dilcher}.

\end{proof}

Let $q$ be a rational number and put
$$f_{\ell,q}(x)=x(x+1)\cdots (x+\ell-1)+q.$$

\begin{lemma} \label{Rl+q}
Suppose that $\ell\geq 3$. Then $f_{\ell,q}(x)$ has at least three simple zeros apart from the cases $f_{4,1}(x)=x(x+1)(x+2)(x+3)+1$ and $f_{4,-\frac{9}{16}}(x)=x(x+1)(x+2)(x+3)-\frac{9}{16}$.
\end{lemma}

\begin{proof}[Proof of Lemma \ref{Rl+q}] This is a reformulation of Theorem 2 in \cite{yuan}.
\end{proof}

Our next auxiliary result is an easy consequence of an effective theorem concerning the $S$-integer solutions of so-called hyperelliptic equations.

\begin{lemma}\label{lem:hyper}
Let $f(x)$ be a polynomial with rational coefficients and with at least two distinct zeros and $u,v$ be fixed positive rational numbers. Then the equation
$$f\left(\frac{x}{u}\right)=vy^z$$
in integers $x, y>1$ and $z>1$ implies $z<C_3$, where $C_3$. Further, if the polynomial $f$ has at least two simple zeros, then all the solutions $x$ and $y$ of the equation
$$f\left(\frac{x}{u}\right)=vy^m, m\geq 3$$
satisfy $\max(|x|,y)<C_4$, and if $f$ possesses at least three simple zeros then all the solutions $x.y$ of the equation
$$f\left(\frac{x}{u}\right)=vy^2$$
are bounded by $C_5$. Here $C_3,C_4$ and $C_5$ are effectively computable constants depending on  the parameters of $f, u$ and $v$.
\end{lemma}

\begin{proof}[Proof of Lemma \ref{lem:hyper}]
This lemma is an easy consequence of a classical theorem of Schinzel and Tijdeman \cite{ST} and the main result of \cite{brindza}.
\end{proof}

The ineffective statement of this paper is  mainly based on the following lemma, which is analogous to Theorem 4.4 in \cite{BBKPT}.

\begin{lemma} \label{lemma:mainineff}
Let $k\geq 2$ be a rational integer with $k\notin\{3,5\}$. There exist no polynomial $p(x)$ and  $\alpha, \beta, \gamma, \delta\in \mathbb C$ such that
$$S_{a,b}^{k}(x)=R_{c}^{\ell}(p(x)\sqrt{\alpha x^2+\beta x+\gamma}+\delta).$$
\end{lemma}

To prove this lemma we need the next result.

\begin{lemma} \label{lemma:aux}
Assume that $f(x), g(x)\in \mathbb Q[x]$ and that $f(x)=g(\lambda x+\nu)$. Further, suppose that all the zeros of $g(x)$ are rational and that $f(x)$ vanishes at $\beta\in \mathbb Q$ but it is not of the form $h((x-\beta)^d)$, where $h(x)\in \mathbb Q[x]$ and $d>1$. Then $\lambda, \nu\in \mathbb Q$.
\end{lemma}

\begin{proof} This is Lemma 4.5 in \cite{BBKPT}.
\end{proof}

\begin{proof}[Proof of Lemma \ref{lemma:mainineff}]
We have
$$R_{c}^{\ell}(p(x)\sqrt{\alpha x^2+\beta x+\gamma}+\delta)=c^{\ell}R_{\ell}\left((p(x)/c\sqrt{\alpha x^2+\beta x+\gamma}+\delta/c\right),$$
so up to replacing $p(x)$ and $\delta$ by $p(x)/c$ and $\delta/c$, we may work with the polynomial $c^{\ell}R_{\ell}(x)$ instead of the polynomial $R_{c}^{\ell}(x)$. We follow the proof of Theorem 4.4 in \cite{BBKPT}. We start with the particular case for $k, \ell \geq 2$ there exist no polynomial $p(x)$ such that
$$S_{a,b}^{k}(x)=c^{\ell} R_{\ell}(p(x)).$$
Assume on the contrary, Lemma \ref{thm:BPS} implies that $\deg p(x)\leq 2$. Suppose first that $\deg p(x)=1$. Then $p(x)=\lambda x + \nu$ and $\ell=k+1$. Suppose first that $\frac{b}{a}=\frac{1}{2}$ and that $k$ is odd. Then $b=1, a=2$ and
$$S_{2,1}^{k}(x)=\frac{2^k}{k+1}\left(B_{k+1}\left(x+\frac{1}{2}\right)-B_{k+1}\left(\frac{1}{2}\right)\right)=$$
$$\frac{2^k}{k+1}\left(x^{k+1}-\frac{(k+1)k}{24}x^{k-1}+\frac{(k+1)k(k-1)(k-2)}{384}x^{k-3}+\ldots \right)$$
for all $k\geq 5$. The zeros of $R_{\ell}(\lambda x+\nu)$ are
$$-\frac{j+\nu}{\lambda}$$
for $j=0, \ldots, \ell$, and their sum must be $0$ because $x^k$ appears with coefficient equal to zero in  $S_{2,1}^{k}(x)$, therefore
$$0=-\frac{1}{\lambda}\sum_{j=0}^{k}(j+\nu),$$
so
$$\nu=-\frac{k}{2}=-\frac{\ell-1}{2}.$$
Thus we get that
$$S_{2,1}^{k}(x)=c^{k+1}R_{k+1}\left(\lambda x-\frac{k}{2}\right)=c^{k+1}\widehat{R}_{(k+1)/2}((\lambda x)^2)=$$
$$=c^{k+1}\left((\lambda x)^2-\frac{1}{4}\right)\left((\lambda x)^2-\frac{9}{4}\right)\cdots \left((\lambda x)^2-\frac{k^2}{4}\right)=$$
$$=c^{k+1}\left( (\lambda x)^{k+1}-\frac{k(k+1)(k+2)}{24}(\lambda x)^{k-1}+\right.$$
$$\left. +\frac{k(k^2-1)(k^2-4)(5k+12)}{5760}(\lambda x)^{k-3}+\ldots \right).$$
Identifying the first three nonzero coefficients above, we get
$$\frac{2^k}{k+1}=(c\lambda)^{k+1},$$
$$\frac{2^k k}{24}=\frac{c^{k+1}\lambda^{k-1}k(k+1)(k+2)}{24},$$
$$\frac{2^k k(k-1)(k-2)}{384}=\frac{c^{k+1}\lambda^{k-3}k(k^2-1)(k^2-4)(5k+12)}{5760}.$$
Dividing the first equation by the second one we have
$$\lambda^2=k+2$$
and dividing the second equation by the third one we obtain
$$\lambda^2=\frac{5k+12}{15},$$
giving $k=-1.8$, contradiction.

From now on, we assume that either $b/a\neq 1/2$ or $b/a=1/2$ but $k$ is even. Then the argument in \cite{BBKPT} applies. Namely, $S_{a,b}^{k}(x)$ has a zero at $x=0$ and it is not of the form $h(x^d)$ for some $d>1$ and polynomial $h(x)$ by Lemma \ref{thm:BPS}, so $\lambda, \nu \in \mathbb Q$ by Lemma \ref{lemma:aux}. In particular, all the zeros of the polynomial $R_{k+1}(\lambda x+\nu)$ are real. By Lemma \ref{eff:1}, we deduce that $k\leq 5$. So, we have to check the impossibility of the identity
$$S_{a,b}^k(x)=c^{k+1}R_{k+1}(\lambda x+\nu)$$
for some $a,b,c\in \mathbb N$ with $\gcd(a,b)=1$, $\lambda, \nu\in \mathbb Q$ and $k\in \{2,3,4,5\}$. We give the details only for $k=2$ the calculations in other cases are very similar and we leave them to the reader.
For $k=2$ we have
$$S_{a,b}^{2}(x)=\frac{a^2}{3}x^3+\frac{a(2b-a)}{2}x^2+\left(\frac{a^2}{6}-ab+b^2\right)x$$
and
$$c^3R_3(\lambda x+\nu)=c^3\lambda^3 x^3+3c^3\lambda^2(\nu+1)x^2+c^3(2\lambda+6\lambda\nu+3\lambda \nu^2)x+c^3(\nu^3+3\nu^2+2\nu).$$

On comparing the corresponding coefficients we get
\begin{equation} \label{comp:1}
\frac{a^2}{3}=c^3\lambda^3
\end{equation}

\begin{equation} \label{comp:2}
\frac{a(2b-a)}{2}=3c^3\lambda^2 (\nu+1)
\end{equation}

\begin{equation} \label{comp:3}
\frac{a^2}{6}-ab+b^2=c^3\lambda (3\nu^2+6\nu+2)
\end{equation}
and

\begin{equation} \label{comp:4}
0=c^3 \nu (\nu+1)(\nu+2).
\end{equation}

From (\ref{comp:4}) we obtain $\nu \in \{0,-1,-2\}$. Suppose first that $\nu=0$. Then dividing (\ref{comp:1}) by (\ref{comp:2}) and dividing (\ref{comp:2}) by (\ref{comp:3}) we get
$$\lambda=\frac{2a}{2b-a}$$
and
$$\lambda=\frac{a(2b-a)}{\frac{a^2}{2}-3ab+3b^2}.$$
These relations yield
$$a^2-6ab+6b^2=(2b-a)^2$$
and
$$2b(b-a)=0,$$
thus $a=b$ or $b=0$, a contradiction. Now assume that $\nu=-1$. Then from (\ref{comp:2}), $a(2b-a)=0$, we get a contradiction again. Finally, if $\nu=-2$, using the previous argument, and obtaining  $3b^2+2ab=0$ we arrive at a contradiction.

We now assume that that $\deg p(x)=2$, in which case $k+1=2\ell$. By Lemma \ref{thm:BPS}, the decomposition $S_{a,b}^k (x)=c^{\ell}R_{\ell}(p(x))$ is equivalent to
$$S_{a,b}^{k} (x)= \widehat{S}_{(k+1)/2} \left(\left(x+\frac{b}{a} - \frac{1}{2}\right)^2\right)$$
which means that
$$p(x)=\lambda\left(x+\frac{b}{a} - \frac{1}{2}\right)^2 + \nu \quad \text{and} \quad  \widehat{S}_{(k+1)/2}(x)=c^{\ell} R_{\ell}(\lambda x+\nu).$$
If $\ell=2$, we get $k=3$, however this is an easily excludable case. Thus we may assume that $\ell\geq 3$. The polynomial $\widehat{S}_{k}(x)$, vanishes at $x_0=(1/2-b/a)^2$, because
$$\widehat{S}_{m}\left(\left(\frac{1}{2}-\frac{b}{a}\right)^2\right)=S_{a,b}^{2m-1}\left(1-\frac{2b}{a}\right)=$$
$$=\frac{a^{2\ell-1}}{2\ell}\left(B_{2\ell}\left(1-\frac{2b}{a}+\frac{b}{a}\right)-B_{2\ell}\left(\frac{b}{a}\right)\right)=0,$$
where we used the fact that $B_{2\ell}(1-y)=B_{2\ell}(y)$ with $y=b/a$. This polynomial is not of the form $h((x-x_0)^d)$ for some $d>1$ by the argument from the footnote of page 181 on \cite{BBKPT}. Indeed, if it were, by the indecomposability of $\widehat{S}_{\ell}(x)$ (see Lemma \ref{thm:BPS}), we would get that
$$h(x)=\frac{a^{2\ell-1}}{(x-x_0)^m},$$
so
$$\frac{a^{2\ell-1}}{2\ell}\left(B_{2\ell}\left(x+\frac{b}{a}-B_{2\ell}\left(\frac{b}{a}\right)\right)\right)=\frac{a^{2\ell-1}}{2\ell}\left(\left(x+\frac{b}{a}-\frac{1}{2}\right)^2-x_0\right)^{\ell},$$
so
$$B_{2\ell}(x)=(x^2-x_0)^{\ell}+C,$$
where
$$C=\frac{2\ell}{a^{2\ell-1}}B_{2\ell}\left(\frac{b}{a}\right).$$
Taking the derivative in the above formula and using the fact that $k\geq 3$, we conclude that $\pm \sqrt{x_0}$ are double roots of $B_{2\ell}'(x)=2mB_{2\ell-1}(x)$, which is impossible by Lemma \ref{eff:1}, part (iii). Hence, $\lambda, \nu \in \mathbb Q$. It remains to identify coefficients. It is easy to see that the polynomial $\widehat{S}_{\ell}(x)$ and the polynomial $\widetilde{B_{\ell}}(x)$ of \cite{BBKPT} are related via the formula
$$\widehat{S}_{\ell}(x)=\frac{a^{2\ell-1}}{2\ell}\widetilde{B_{\ell}}(x)+D,$$
with
$$D=\frac{a^{2\ell-1}}{2\ell}\left(B_{2\ell}-B_{2\ell}\left(\frac{b}{a}\right)\right).$$
Thus, we get, from a previous calculation (with the change of variable $k+1=2\ell$),
$$\widehat{S}_{\ell}(x)=\frac{a^{2\ell-1}}{2\ell}\left(x^\ell-\frac{2\ell(2\ell-1)}{24}x^{\ell-1}+\right.$$
$$\left.+\frac{2\ell(2\ell-1)(2\ell-2)(2\ell-3)}{384}x^{\ell-2}+\ldots \right).$$
Writing
$$c^{\ell}R_{\ell}(\lambda x+\nu)=c^{\ell}\left((\lambda x+\nu)^{\ell}+\frac{\ell(\ell-1)}{2}(\lambda x+\nu)^{\ell-1}+\right.$$
$$\left. +\frac{\ell(\ell-1)(2\ell-1)}{6}(\lambda x+\nu)^{\ell-2}+\ldots \right),$$
and identifying the corresponding coefficients, we get
$$\frac{a^{2\ell-1}}{2\ell}=c^\ell \lambda^m;$$
$$-\frac{a^{2\ell-1}(2\ell-1)}{24}=c^{\ell} \lambda^{\ell-1}\ell\left(\nu+\frac{\ell-1}{2}\right);$$
and
$$\frac{a^{2\ell-1}(2\ell-1)(2\ell-2)(2\ell-3)}{384}=$$
$$\frac{c^{\ell}\lambda^{\ell-2}\ell(\ell-1)}{2}\left(\nu^2+(\ell-1)\nu+\frac{2\ell-1}{3} \right).$$
Taking ratios of the first two equations and then the next two equations, we get
$$\frac{\lambda}{\nu+(\ell-1)}=-\frac{12}{2\ell-1};$$
$$\frac{\lambda(\nu+(\ell-1)/2}{\nu^2+(\ell-1)\nu+(2\ell-1)/3}=-\frac{4}{2\ell-3}.$$
Dividing the second equation above equation by the first, we get
$$\frac{(\nu+(\ell-1)/2)^2}{(\nu+(\ell-1)/2)^2-(3\ell^2-14\ell+7)/12}=\frac{2\ell-1}{3(2\ell-3)}.$$
This gives
$$\frac{x^2}{x^2-(3\ell^2-14\ell+7)/12}=\frac{2\ell-1}{3(2\ell-3)}$$
with $x=\nu+(\ell-1)/2$, so
$$(\ell-2)x^2=-\frac{(2\ell-1)(3\ell^2-14\ell+7)}{12}.$$
This can be checked to be false for $\ell=2,3,4$ and for $\ell\geq 5$, the left-hand side is positive and the right-hand side is negative.
This shows that indeed it is not possible that $S_{a,b}^k(x)=c^{\ell}R_{\ell}(p(x))$ for some polynomial $p(x)$.
Now we can prove the theorem in its full generality. Assume that
$$r(x)=\alpha x^2+\beta x+\gamma$$
is not a complete square, otherwise $p(x)\sqrt{r(x)}+\delta$ is a polynomial, case which has already been treated. The argument from \cite{BBKPT} applies to say that
$$c^{\ell}R_{\ell}(p(x)\sqrt{r(x)}+\delta)=c^{\ell} r(x)^{\ell/2}p(x)^{\ell}+$$
$$+c^{\ell} r(x)^{(\ell-1)/2}p(x)^{\ell-1}\left(\ell\delta+\frac{\ell(\ell-1)}{2}\right)+\ldots $$
is a polynomial so $\ell$ must be even. Furthermore,
$$\ell\delta+\frac{\ell(\ell-1)}{2}=0,$$
that is $\delta=-\frac{\ell-1}{2}$.
But then
$$R_{\ell}(p(x)\sqrt{r(x)}+\delta)=R_{\ell}\left(p(x)\sqrt{r(x)}-\frac{\ell-1}{2}\right)=\widehat{R}_{\ell/2}(r(x)p(x)^2).$$
Thus, $S_{a,b}^k(x)=\widehat{R}_{\ell/2}(\widetilde{p}(x))$, where $\widetilde{p}(x)=r(x)p(x)^2$. The case $\ell=2$ leads to
$$S_{a,b}^k(x)=cr(x)p(x)^2-\frac{c}{4},$$
so
$$\frac{a^k}{k+1}\left( B_{k+1}\left(x+\frac{b}{a}\right)-B_{k+1}\left(\frac{b}{a}\right)   \right)=cr(x)p(x)^2-\frac{c}{4},$$
so
$$B_{k+1}(x)=\frac{c(k+1)}{a^k}r\left(x-\frac{b}{a}\right)p\left(x-\frac{b}{a}\right)^2+\left(B_{k+1}\left(\frac{b}{a}\right)  -\frac{c(k+1)}{4a^k}\right).$$
By Lemma \ref{eff:1}, we get $k\in \{3,5\}$, however, these cases are excluded by the condition of our lemma.

So, it must be the case that $\ell\geq 4$. We have $S_{a,b}^k(x)=\widehat R_{\ell/2}(\widetilde{p}(x))$. By Lemma \ref{thm:BPS} and the fact that $r(x)$ is not a complete square, it follows that in fact $r(x)$ is a linear polynomial. Say $r(x)=\lambda x+\nu$. Assume first that $b/a=1/2$ and $n$ is even. We then get
$$\widehat{R}_{\ell/2}(r(x))=S_{a,b}^k(x)=\widehat{S}_{(k+1)/2}\left(\left(x+\frac{b}{a}-\frac{1}{2}\right)^2\right),$$
with a linear polynomial $r(x)$, contradicting the indecomposability of $\widehat{R}_{\ell/2}(x)$, see Lemma  \ref{lemma:Rlx}. So either $b/a\neq 1/2$ or $b/a=1/2$ but $k$ is not even. Then $S_{a,b}^{k}(x)$ has $x=0$ as a zero but it is not of the form $h(x^d)$ for any $d>1$ by Lemma \ref{thm:BPS} and $\widehat{R}_{\ell}(x)$ has rational zeros. So, from Lemma \ref{lemma:aux}, $\lambda$ and $\nu$ are rational. In particular, all zeros of $\widehat{R}_{\ell/2}(r(x))$ are real. Thus, $S_{a,b}^{k}$ has only real roots showing that $k\in \{2,3,4,5\}$. Considering these small cases, on comparing the corresponding coefficients we obtain a contradiction.
\end{proof}

We will introduce some notation to recall the finiteness criterion by Bilu and Tichy. In what follows $\alpha$ and $\beta$ are nonzero rational numbers, $\mu,\nu$ and $q$ are positive integers, $p$ is a nonnegative integer and $\nu(X)\in \mathbb Q[X]$ is a nonzero polynomial (which may be constant).

A standard pair of the first kind is $(X^q,\alpha X^{p}\nu(X)^q)$ or switched, $(\alpha X^{p}\nu(X)^q, X^q)$, where $0\leq p<q, (p,q)=1$ and $p+\deg \nu(X)>0$.

A standard pair of the second kind is $(x^2,(\alpha x^2+\beta)\nu(x)^2)$ (or switched).

Denote by $D_{\mu}(x,\delta)$ the $\mu$th Dickson polynomial, defined by the functional equation $D_{\mu}(z+\delta/z,\delta)=z^{\mu}+(\delta/z)^{\mu}$or by the explicit formula
$$D_{\mu}(x,\delta)=\sum_{i=0}^{[\mu/2]}d_{\mu,i}x^{\mu-2i},$$
with
$$d_{\mu,i}=\frac{\mu}{\mu-i}{\binom{\mu-i}{i}}(-\delta)^i.$$
A standard pair of the third kind is $D_{\mu}(x,\alpha^{\nu}),D_{\nu}(x,\alpha^{\mu}$, where $\gcd(\mu,\nu)=1$.

A standard pair of the fourth kind is $\left(\alpha^{-\mu/2} D_{\mu}(x,\alpha), -\beta^{-\nu/2}D_{\nu}D_{\nu}(x,\beta)\right),$ where $\gcd(\mu,\nu)=2$.

A standard pair of the fifth kind is $((\alpha x^2-1)^3, 3x^4-4x^3)$ (or switched).

\begin{lemma} \label{lemma:BT}
Let $Rx),S(x)$ be nonconstant polynomials such that the equation $R(x)=S(y)$ has infinitely many solutions in rational integers $x,y$. Then $R(x)=\phi(f(\kappa(x)))$ and $S(x)=\phi(g(\lambda(x)))$ where $\kappa(x),\lambda(x)\in \mathbb Q[x]$ are linear polynomials, $\phi(x)\in \mathbb Q[x]$, and $(f(x),g(x))$ is a standard pair.
\end{lemma}

We need the analogs of Lemmata 5.2 and 5.3 in \cite{BBKPT}. Both of them follow immediately from the analogous results in \cite{BBKPT}, so their proofs are omitted.

\begin{lemma} \label{lemma:npower}
None of the polynomials $S_{a,b}^n(a_1x+a_0)$ or $c^mR_m(b_1x+b_0)$ or  is of the form $e_1x^q+e_0$ with $q\geq 3$.
\end{lemma}

\begin{lemma} \label{lemma:ndickson}
The polynomial $S_{a,b}^{n}(a_1 x+a_0)$ is not of the form $e_1D_{t}(x,\delta)+e_0$, where $D_t(x,\delta)$ is the Dickson polynomial with $t>4$ and $\delta\in \mathbb Q\setminus \{0\}$.
\end{lemma}

\section{Proofs of the Theorems}

\begin{proof}[Proof of  Theorem \ref{thm:eff1}]

If $b=0$ we essentialy obtain the original Sch\"affer equation so in the sequel we assume $b\neq 0$.

For $k\in \{2,4\}$ or $k\geq 6$, our theorem is an easy consequence of (\ref{eq:mainI}), Lemmata \ref{eff:1} and \ref{lem:hyper}. Suppose that $k=1$ and consider the equation
$$S_{a,b}^1 (x) =\frac{1}{2}x(ax+2b-a)=y^{\ell}.$$

Since $(a,b)=1$, one can see that the quadratic polynomial on the left hand side has two simple zeros apart from the case $a=2,b=1$. For $k=3$ and $k=5$ the discriminant of $S_{a,b}^k (x)$
is
$$\frac{1}{256}a^6b^4(a-b)^4(a-2b)^2(a^2+4ab-4b^2)$$
and
$$\frac{1}{967458816}a^{20}b^4(a-b)^4(a-2b)^2(a^2+3ab-3b^2)^4(a^2-6ab+6b^2)^2\times$$
$$\times (a^2+4ab-4b^2)(a^2+2ab-2b^2)^2(3a^4+12a^3b+4a^2b^2-32ab^3+16b^4), $$
respectively.

We have two critical cases (i. e. $\frac{a}{b}$  is a rational zero of discriminants,) $a=b=1$ and $a=2,b=1$. In the first case
$$S_{1,1}^{k}(x)=S_{k}(x+1),$$
where $S_k(x)$ denotes the usual Sh\"affer's sum of $k$th powers. If $a=2, b=1$ then we get
$$S_{2,1}^{3}(x)=2x^4-x^2\,\,\mbox{and}\,\, S_{2,1}^{5}(x)=\frac{1}{3}x^2(16x^4-20x^2+7),$$
and these polynomials have two and four simple zeros, respectively.

\end{proof}

\begin{proof}[Proof of Theorem \ref{thm:eff2}]
First we consider our equation
$$
S_{a,b}^{k}(x)=R_{c}^{\ell}(y)
$$
in integers $x$ and $y$, where $\ell\geq 2$ and $k\in\{1,3\}$. Formulas (\ref{eq:mainI}) and (\ref{eq:obser})
give
$$ R_{c}^{\ell}(y)=c^{\ell}R_{\ell}\left(\frac{y}{c}\right)=\frac{1}{2}x(ax+2b-a),$$
$$8ac^{\ell}R_{\ell}\left(\frac{y}{c}\right)+(2b-a)^2=(2ax+2b-a)^2,$$
and
$$ R_{c}^{\ell}(y)=c^{\ell}R_{\ell}\left(\frac{y}{c}\right)=\frac{1}{4}x(ax+2b-a)\times$$
$$\times (a^2x^2+(2ba-a^2)x+2b^2-2ab),$$
$$4ac^{\ell}R_{\ell}\left(\frac{y}{c}\right)=X(X+2b^2-2ab)=(X+(b^2-ab))^2-(b^2-ab)^2,$$
where $X=a^2x^2+(2ab-b^2)x$, respectively,  and Lemmas \ref{Rl+q} and \ref{lem:hyper} complete the proof for $(k,\ell)\notin \{(1,2),(3,2),(1,4),(3,4)\}$.

Now, if $\ell\in\{2,4\}$ we have
$S_{a,b}^{k}(x)=y(y+c), 4S_{a,b}^{k}(x)+c^2=(2y+c)^2$
and
$$S_{a,b}^{k}(x)=y(y+c)(y+2c)(y+3c)=(y^2+3cy+c^2)^2-c^4,$$
respectively, and our result is proved by  (\ref{eq:mainI}) and Lemmas \ref{eff:1} and \ref{lem:hyper} for
$$(k,\ell)\notin \{(1,2),(3,2),(1,4),(3,4),(5,2),(5,4)\}.$$
\end{proof}

\begin{proof}[Proof of Theorem \ref{thm:eff3}]
For $(k,\ell)=(1,4)$ we get
$$\frac{1}{2}x(ax+2b-a)=y(y+c)(y+2c)(y+3c),$$
and
$$8ac^4\frac{y}{c}\left(\frac{y}{c}+1  \right)\left(\frac{y}{c} +2  \right)\left( \frac{y}{c}+3 \right)=(2ax+2b-a)^2-(2b-a)^2.$$
Since $\frac{(2b-a)^2}{8ac^4}$ is non-negative, we cannot guarantee three simple zeros when this fraction is 1 (cf. Lemma \ref{Rl+q}). If
 $(2b-a)^2=8ac^4$, we have $a=2$. Indeed, by the parities $a\neq 1$. Denote by $p$ an arbitrary prime divisor of $a$, so $p|a$ and thus $p|2b$, and $p=2$. Now, if $a=2^{\alpha}$, where $\alpha\geq 2$, then $\mbox{ord}_2(2b-a)=1$ which is a contradiction, so $a=2$ and  $(b-1)^2=4c^4$. For $(k,\ell)=(3,4)$ we can apply a very similar argument, and here
$$\frac{(b^2-ab)^2}{4ac^4}=1,$$
and this yields $a=1,\left(b(b-1)\right)^2=4c^4$.

For $(k,\ell)=(3,2),(5,2)$ and $(5,4)$ we follow the same idea, and give the details only for case $(k,\ell=(3,2)$.
Consider the equation
$$S_{a,b}^{3}(x)=\frac{a^3}{4}\left(B_4\left(x+\frac{b}{a}\right)-B_4\left(\frac{b}{a}\right)\right)=R_{c}^{2}(y)=y(y+c),$$
and
$$a^3\left(B_4\left(x+\frac{b}{a}\right)-B_4\left(\frac{b}{a}\right)+\frac{c^2}{a^3}\right)=(2y+c)^2.$$
Finally, Lemma \ref{eff:1} completes the proof.
\end{proof}

\begin{proof}[Proof of Theorem \ref{thm:main}] We follow Section 5.3 of \cite{BBKPT}. In view of the small cases treated and
the fact that we have proved the analog of Theorem 4.4 in \cite{BBKPT}, the argument
from Page 184 shows that we may assume that $(f(x), g(x))$ do not form
a pair of second or fifth kind. If it is of the first kind, we get the same
contradiction based on Lemma  \ref{lemma:npower}, and if it is of the fourth kind, we get
again the contradiction based on Lemma \ref{lemma:ndickson}. So, we only need to revisit the
argument in \cite{BBKPT} for the pairs of the third kind. For this, we just notice that,
with the notations from there, all coefficients $s_i$ get multiplied by $a^n = a_2$
(except for the last one which also gets shifted but hopefully we shall not
get to it), and all the coefficients $r_j$ get multiplied by $c_m$. So, the analogs of
(26)-(29) in \cite{BBKPT} become
$$s_3=\frac{b_{1}^{3}a^2}{3}=e_1,$$
$$s_1=-\frac{b_1a^2}{24}=-3e_1\alpha^m,$$
$$r_m=a_{1}^{m}c^m=e_1,$$
$$r_{m-2}=-a_{1}^{m-2}c^m\frac{m(m-1)(m+1)}{24}=-e_1m\alpha^3.$$
So we get
$$\alpha^m=\frac{b_{1}^{-2}}{24}, \alpha^3=a_{1}^{-2}\frac{m^2-1}{24}, c^ma_{1}^{m}=\frac{a^2b_{1}^{3}}{3}.$$
Hence,
$$\frac{b_{1}^{-6}}{24^3}=a_{1}^{-2m}\left(\frac{m^2-1}{24}\right)^m=(a^2c^{-m})^{-2}9b_{1}^{-6}\left(\frac{m^2-1}{24}\right)^m,$$
giving
$$\frac{1}{2^9 3^5}=\left(\frac{c^m}{a^2}\right)^2\left(\frac{m^2-1}{24}\right)^m.$$
If $m$ is even, the number on the right above is a square of a rational number, whereas the number on the left is not. So, $m$ is odd. Now, we get
$$\frac{1}{2^9 3^5}=\left(\frac{c^m}{a^2}\right)^2\left(\frac{m^2-1}{24}\right)^{m-1}\left(\frac{m^2-1}{24}\right),$$
or
$$\frac{1}{m^2-1}=2^6 3^4\left(\frac{c^m}{a^2}\right)^2\left(\frac{m^2-1}{24}\right)^{m-1},$$
and the right-hand side above is a square of a rational number, therefore so
is the left-hand side, so $m^2-1$ is a square, contradiction.
The theorem is proved.
\end{proof}

\section*{Acknowledgements}

The research of the authors was supported in part by the E\"otv\"os Lor\'and Research Network (ELKH) and by the NKFIH grants ANN130909 and K128088.

\bibliographystyle{amsplain}

\end{document}